\RequirePackage[ngerman,english]{babel}
\documentclass[12pt,a4paper]{amsart}

\usepackage{amsfonts, amsmath, amssymb, amsthm}

\usepackage{german}
\usepackage[ngerman,english]{babel}

\usepackage{setspace}

\theoremstyle{plain}
\newtheorem{thm}{Theorem}[section]
\newtheorem*{thm*}{Theorem}
\newtheorem{lem}[thm]{Lemma}
\newtheorem{prop}[thm]{Proposition}
\newtheorem{cor}[thm]{Corollary}

\theoremstyle{definition}

\theoremstyle{remark}

\newtheorem{qst}[thm]{Question}

\usepackage{algorithmic,algorithm}

\usepackage{txfonts}
\usepackage{amssymb}
\usepackage{mathbbol}
\usepackage{ae}
\usepackage{bbm}
\usepackage[mathscr]{eucal} 
\usepackage{xspace}
\usepackage{url}

\usepackage{color}
\usepackage{graphicx}
\usepackage{overpic}
\usepackage{subfigure}
\usepackage{booktabs}
\usepackage{paralist}

\usepackage[a4paper,scale=0.8, marginratio={1:1, 9:10}, ignoreall]{geometry}

\numberwithin{equation}{section}
\numberwithin{figure}{section}

\newcommand\st{\; | \;}

\newcommand\cut{\cap}

\newcommand\union{\cup}
\newcommand\Union{\bigcup}

\newcommand\cP{{\mathcal P}}
\newcommand\ncP{{\mathcal P}^\star}

\newcommand\cA{{\mathcal A}}
\newcommand\cB{{\mathcal B}}
\newcommand\cF{{\mathcal F}}

\newcommand\cT{{\mathcal T}}

\newcommand\cS{{\mathcal S}}
\newcommand\cC{{\mathcal C}}
\newcommand\cD{{\mathcal D}}

\newcommand\NN{{\mathbb N}}
\newcommand\RR{{\mathbb R}}

\newcommand\SetOf[2]{\left\{#1\vphantom{#2}\colon\,\vphantom{#1}#2\right\}}
\newcommand\smallSetOf[2]{\{#1\colon\,#2\}}

\newcommand\arrow{\vec}

\newcommand\conv{\operatorname{conv}}

\newcommand\relint{\operatorname{relint}}

\newcommand\card[1]{\left|#1\right|}

\providecommand{\nonnegRR}{{\RR_{\ge0}}}

\providecommand{\subdiv}{{\Sigma}}

\providecommand{\envelope}[2]{{{\mathscr{E}}_{#1}(#2)}}
\providecommand{\lift}[2]{{{\mathscr{L}}_{#1}(#2)}}
\providecommand{\tightspan}[2]{{{\mathscr{T}}_{#1}(#2)}}

\providecommand{\subdivision}[2]{{\subdiv_{#1}(#2)}}

\providecommand{\Asubdivision}[2]{{\subdiv_{#1}(#2)}}

\providecommand{\Hypersimplex}[2]{{\Delta(#1,#2)}}

\providecommand{\scp}[2]{\langle{#1},{#2}\rangle}

\renewcommand{\phi}{\varphi}

\DeclareMathOperator{\SP}{SP}
\providecommand{\divTS}[1]{{\hat T}_{#1}}
\providecommand{\divTSm}[1]{{T}_{#1}}
\providecommand{\divP}[1]{{\hat P}_{#1}}
\providecommand{\divPm}[1]{{P}_{#1}}

\usepackage[T1]{fontenc}
\usepackage[utf8]{inputenc}

\begin{document}

\title{Trees, Tight-Spans and Point Configurations}
\author{Sven Herrmann \and Vincent Moulton}
\date{\today}
\address{School of Computing Sciences, University of East Anglia, Norwich, NR4 7TJ, UK}

\thanks{The first author was supported by a fellowship within the Postdoc"=Programme of the German Academic Exchange Service (DAAD) and thanks the UEA School of Computing Sciences
for hosting him during the writing of this paper.}

\begin{abstract}
Tight-spans of metrics were first introduced by Isbell
in 1964 and rediscovered and studied by others, most 
notably by Dress, who gave them this name. Subsequently, it was
found that tight-spans could be defined for
more general maps, such as directed metrics and 
distances, and more recently for diversities. 
In this paper, we show that all of these tight-spans 
as well as some related constructions can 
be defined in terms of point configurations. This 
provides a useful way in which to study these 
objects in a unified and systematic way. We also 
show that by using point configurations we can recover 
results concerning one-dimensional tight-spans
for all of the maps we consider, 
as well as extend these and other results to 
more general maps such as symmetric and unsymmetric maps.
\end{abstract}


\maketitle

\section{Introduction}
Let $V$ be a real vector space with
standard scalar product $\scp \cdot \cdot$ with respect
to some fixed basis $B$ (i.e., $\scp vw=\sum_{b\in B} \lambda_b \mu_b$ if $v=\sum_{b\in B}\lambda_b b, w=\sum_{b\in B}\mu_b b$). 
A \emph{point configuration} $\cA$
in $V$ is a finite subset of $V$; for technical reasons
we shall assume that the affine hull of any 
such configuration has codimension $1$.
Given a function $w: \cA \to \RR$, 
we define the \emph{envelope} of~$\cA$ with respect to $w$
to be the polyhedron
\[
\envelope{w}{\cA} \ = \ \SetOf{x\in V}
{\scp {a} x \ge -w \text{ for all } a\in\cA},
\]
and the \emph{tight-span} $\tightspan w \cA$ of~$\cA$
to be the union
of the bounded faces of $\envelope{w}{\cA}$.
Tight-spans of point configurations
were introduced in \cite{MR2502496}
for vertex sets of polytopes, 
as a tool for studying subdivisions of polytopes.
Even so, they first appeared several years ago
in a somewhat different guise. 

More specifically, let $X$ be a finite set,
$V = \RR^X$ be the vector space
of functions $X\to \RR$ and, for $x\in X$, $e_x$ denote the elementary
function assigning $1$ to $x$ and $0$ to all other $y\in X$.
In addition, let $D$ be a metric on $X$, that is, a symmetric map
on $X \times X$ that vanishes on the
diagonal and satisfies the triangle inequality. Then,
as first remarked by Sturmfels and  Yu~\cite{MR2097310},
by setting $w(e_x+e_y)=-D(x,y)$,
the tight-span $\tightspan w {\bar\cA(X)}$ of
$\bar \cA(X)=\smallSetOf{e_x+e_y}{x,y\in X, x\not=y}$
is nothing other than the \emph{injective hull} of $D$
that was first introduced by Isbell~\cite{MR82949}
and subsequently rediscovered by Dress~\cite{MR753872} (who
called it the \emph{tight-span of $D$}),
as well as Chrobak and Larmore~\cite{MR1105939,MR1258238}.

Since its discovery by Isbell, the tight-span of a metric 
on a finite set has been
intensively studied (see, e.g., \cite{DHKMS-book,MR1379369}
for overviews) and various related constructions have been introduced.
These include tight-spans of {\em directed metrics} and
{\em directed distances}~\cite{HiraiKoichi}, tight-spans of 
{\em polytopes} \cite{MR2502496} and 
more recently the tight-span of a 
so-called {\em diversity}~\cite{BT10}.
Note that, in contrast to the tight-span of a metric,
it is not known whether or not 
all of these constructions are necessarily
injective hulls (i.e., injective objects
in some appropriate category), but for 
simplicity we shall still refer to them as tight-spans.
Here we shall show that, as with metrics on finite sets, tight-spans
of directed distances, diversities and some related maps can 
all also be described in terms of
point configurations, providing a useful way
to systematically study these objects.

More specifically, after presenting some preliminary results
concerning point configurations 
in Sections~\ref{sec:point-configurations} and~\ref{sec:splits},
in Section~\ref{sec:symmetric-functions}
we shall show that the tight-span of a distance on $X$
can be defined in terms of the configuration
$\cA(X)=\bar \cA(X)\cup \smallSetOf{2 e_x}{x\in X}
=\smallSetOf{e_x+e_y}{x,y\in X}$  
(Proposition~\ref{prop:ts-dissimilarity}). Also, for $Y$
a finite set with $X\cap Y=\emptyset$, let
$\bar \cB(X,Y)\subseteq \RR^{X\cup Y}$ be
the configuration of all points $e_x+e_y$ with $x\in X$,
$y\in Y$ and $\cB(X,Y)=\bar \cB(X,Y)\cup \smallSetOf{2 e_x}{x\in X\cup 
Y}$. We show that the tight-span of a directed metric (distance) can
be defined in terms of  $\bar \cB(X)=\bar \cB(X,Y)$ or $\cB(X)=\cB(X,Y)$,
where we consider $Y$ as a disjoint copy 
of $X$ (Proposition~\ref{prop:ts-directed}). Using
these point configurations, we will also extend this
analysis to include arbitrary symmetric and even unsymmetric maps
(Section~\ref{sec:non-symmetric}).

In Sections~\ref{sec:div-distances} and 
\ref{sec:div-ts} we shall consider 
tight-spans of diversities, which were recently introduced in 
\cite{BT10}. Using a relationship that 
we shall derive between metrics and diversities, in 
Section~\ref{sec:div-ts}
we show that the tight-span of a diversity on $X$ 
can be expressed in terms of the point configuration
$\cC(X)=\smallSetOf{\sum_{i\in A} e_i}{A\in\cP(X)}$ (the vertices of a cube).
Intriguingly, we also show that a
strongly related object can also be associated to a diversity on $X$
by considering the point configuration 
$\cA(\cP(X)\setminus\{\emptyset\})$ and that,
for a special class of diversities (split system diversities)
this object and the tight-span
are in fact the same (Theorem~\ref{thm:tight-span-equal}).

In addition to providing some new insights 
on tight-spans using point configurations, we shall 
also focus on one-dimensional tight-spans. 
These are important since, for
example, they provide ways to generate phylogenetic trees
and networks (see, e.g., \cite{DHKMS,MR2232989}).
To see why this is the case, note that a one-dimensional
tight-span associated to a point configuration 
$\cA$ and weight function $w$ 
can also be regarded
as a graph, with vertex set equal to that of 
$\envelope{w}{\cA}$ and edge set
consisting of precisely those pairs of vertices
that both lie in a one-dimensional face of
$\envelope{w}{\cA}$. Since the union of bounded
faces of an unbounded polyhedron is contractible
(see, e.g., \cite[Lemma~4.5]{MR2233884}) it follows that
in this case the tight-span is, in fact, a tree. 

The archetypal characterisation
for one-dimensional tight-spans 
was first observed by Dress for metrics \cite{MR753872}:

\begin{thm}[Tree Metric Theorem]
The tight-span of a metric $D$ on a finite set $X$
is a tree if and only if $D$ satisfies
$$
D(x,y)+D(u,v) \le \max\{D(x,u)+D(y,v), D(x,v)+D(y,u)\}
$$
for any $x,y,u,v \in X$.
\end{thm}

In this paper we will use point configurations to give 
various conditions for when tight-spans are
trees in more general settings 
(Theorems~\ref{thm:distance-tree}, \ref{thm:tight-span:tree} 
and \ref{thm:phylogenetic-tree}). This
allows us to recover and extend various
theorems connecting tight-spans and trees that arise
in the literature. We conclude the paper with
a discussion on some possible future directions.

\section{Tight-Spans and Splits of Point Configurations}\label{sec:point-configurations}

In this section, we will recall some definitions and results about tight-spans and splits of general point configurations as well as give some elementary properties of these that we will use later. For details, we refer the reader to~\cite{MR2502496} and \cite[Section~2]{Herrmann09b}. 
First we give a characterisation of the tight-span as the set of minimal elements of the envelope of a configuration if the configuration satisfies certain conditions. These conditions are fulfilled by all of the configurations that we will consider. When tight-spans (of metric spaces, but also of diversities) are considered and thought of in a non-polyhedral way, this characterisation is normally used as definition instead.
 
Now, as in the introduction, let $V$ be a finite"=dimensional vector space. An element of $v\in V$ is called \emph{positive} (with respect to a fixed basis $B$) if in its representation $v=\sum_{b\in B} \lambda^v_b b$ with respect to $B$ one has $\lambda^v_b\geq0$ for all $b\in B$. We have a partial order $\preceq$ on $V$ defined by $v\preceq v'$ if and only if $\lambda^v_b\leq\lambda^{v'}_b$ for all $b\in B$ (or, equivalently, $v'-v$ is positive). For a subset $A\subseteq V$ an element $a\in A$ is called \emph{minimal} if $a\preceq a'$ implies $a=a'$ for all $a'\in A$. The set $A$ is called \emph{bounded from below} if there exists some $M\in \RR$ such that $\lambda^v_b\geq M$ for all $b\in B$ and $v\in A$.

Let now $e\in \NN$ and $\phi: V\to \RR^e$ be a linear map and $b\in \RR^e$. In general, for a polyhedron $P=\smallSetOf{x\in V}{\phi(x)\geq b}$, an element $x\in P$ is contained in a bounded face of $P$ if and only if there does not exist some (non"=trivial) $r\in \smallSetOf{x\in V}{\phi(x)\geq 0}$ (a \emph{ray} of $P$) and some $\lambda\in \RR_{>0}$ with $x-\lambda r\in P$. Note that $P$ is bounded from below if and only if all rays of $P$ are positive. We now give an alternative characterisation for the tight-span.

\begin{lem}\label{lem:minimal}
Let $\cA\subseteq V$ be a configuration of positive points. Then $\tightspan w \cA$ is a subset of the set of minimal elements of $\envelope w \cA$. If, additionally, $\envelope{w}{\cA}$ is bounded from below, then $\tightspan w \cA$ equals the set of minimal elements of $\envelope{w}{\cA}$.
\end{lem}
\begin{proof}
Let $x\in\tightspan w \cA$ be non"=minimal, that is, there exist $b\in B$ and $\lambda\in \RR_{>0}$ such that $x-\lambda b\in P$. By positivity, we have $\scp a{b}\geq 0$ for all $a\in \cA$ and hence  $b$ is a ray of $\envelope{w}{\cA}$ contradicting the assumption $x\in\tightspan w \cA$.

Conversely, let $x\in\envelope{w}{\cA}\setminus\tightspan w \cA$, $r$ be a ray of $\envelope{w}{\cA}$ and $\lambda \in \RR_{>0}$ be such that $x-\lambda r\in \envelope{w}{\cA}$. Since $\envelope{w}{\cA}$ is bounded from below, $r$ is positive and hence $x-\lambda r\preceq x$, so $x$ is not minimal.
\end{proof}

Another simple but useful observation is the following:
\begin{lem}\label{lem:ts-shift}
Let $\cA\subseteq V$ be a point configuration, $w:\cA \to \RR$ a weight function, $v\in V$, and $w'=w+\scp \cdot v$. Then $\tightspan w\cA=\tightspan{w'}\cA+v$.
\end{lem}
\begin{proof}
For all $x\in V$, we have 
\[ \scp {a} x \ge -w'(a)= -(w(a)+\scp av) \Leftrightarrow \scp a{x+v}\geq -w \text{ for all } a\in \cA\,. \]
Hence
\begin{align*}
\envelope{w'}{\cA}&=\SetOf{x\in V}{\scp a{x+v}\geq -w \text{ for all } a\in \cA}\\
&=\SetOf{y-v\in V}{\scp a{y}\geq -w \text{ for all } a\in \cA}=\envelope{w}{\cA}-v\,.
\end{align*}
Obviously, this equation carries over to the unions of the bounded faces, that is, the tight-spans.
\end{proof}



Tight-spans of a point configuration $\cA$ with certain weight functions are closely associated to other objects defined by these weight functions, so-called regular subdivisions which we will define now. The convex hull of $\cA$ is denoted by $\conv \cA$ and the relative interior of a set $A\subseteq V$ is denoted by $\relint A$. For a point configuration $\cA$ we call $F\subseteq \cA$ a \emph{face} of $\cA$ if there exists a supporting hyperplane $H$ of $\conv \cA$ such that $F=\cA\cut H$. An \emph{edge} of $\cA$ is a face of size $2$.  A \emph{subdivision} $\subdiv$ of a point configuration~$\cA$ is a collection of subconfigurations of~$\cA$ satisfying the following three conditions (see~\cite[Section~2.3]{Triangulations}):
\begin{itemize}
\item(SD1)\label{sd:containement} If $F\in\subdiv$ and $\bar F$ is a face of $F$, then $\bar F\in \subdiv$.
\item(SD2)\label{sd:union} $\conv \cA =\Union_{F\in \subdiv} \conv F$.
\item(SD3)\label{sd:intersection} If $F, \bar F\in\subdiv$,$F\not=F'$, then $\relint(\conv F) \cut \relint( \conv \bar F) =\emptyset$.
\end{itemize}

See Figure~\ref{fig:subdiv} for some examples illustrating 
these concepts. A subdivision is a \emph{triangulation} if all faces are simplices, that is, configurations formed by the vertices of a simplex. If $\cA$ is a simplex the only possible subdivision of $\cA$ is the trivial subdivision $\cP(\cA)$ with sole maximal cell being $\cA$ itself.

\begin{figure}
\input{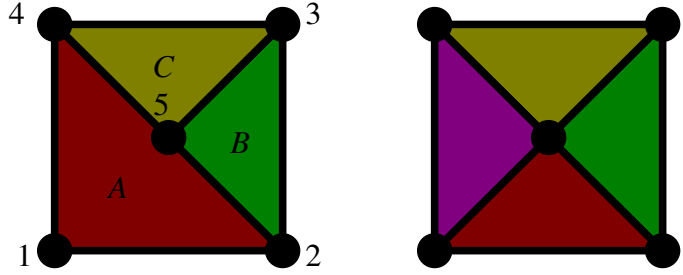}
\caption{Two collections of subconfigurations 
of the five points $\{1,2,3,4,5\}$, as indicated
by the triangles. The collection of subconfigurations
on the left is not a subdivision, 
as it violates (SD3): The intersection of the interior 
of convex hull of the edge $\{2,4\}$ (which is a face of 
the triangle $A=\{1,2,4\}$) and the convex hull of 
the edge $\{4,5\}$ is non-empty. In contrast, the 
collection of subconfigurations on the right 
is a subdivision of $\{1,2,3,4,5\}$. It contains four 
maximal faces (triangles) of cardinality 3 and 
eight edges of cardinality 2.}
\label{fig:subdiv}
\end{figure}

A common way (see \cite[Chapter~5]{Triangulations}) to define such a subdivision is the following: Given a weight function $w:\cA \to \RR$ we consider the \emph{lifted polyhedron}
\begin{equation*}
    \lift{w}{\cA} \ = \conv\SetOf{(w(a),a)}{a\in\cA} \, + \, \nonnegRR (1,0) \subseteq \RR\times V\, .
  \end{equation*}
  
The \emph{regular subdivision} $\Asubdivision w \cA$ of~$\cA$ with respect to $w$ is obtained by taking the configurations $\smallSetOf{b\in\cA}{(w(b),b)\in F}$ for all lower faces $F$ of $\lift w \cA$ (with respect to the first coordinate; by definition, these are exactly the bounded faces). So the elements of $\Asubdivision w \cA$ are the projections of the bounded faces of $\lift w \cA$ to the last $d$ coordinates.

We can now state the relationship between tight-spans and regular subdivisions of point configurations:

\begin{prop}[Proposition~2.1 in \cite{Herrmann09b}]\label{prop:duality}
  The polyhedron $\envelope{w}{\cA}$ is affinely equivalent to the polar dual of the polyhedron $\lift w \cA$.
  Moreover, the face poset of $\tightspan{w}{\cA}$ is anti"=isomorphic to the face poset of
  the interior lower faces (with respect to the first coordinate) of $\lift{w}{\cA}$.
\end{prop}

We shall not define all the notions of this proposition, but note that, as a consequence, the (inclusion) maximal faces of the tight-span $\tightspan{w}{\cA}$ correspond to the (inclusion) minimal interior faces of $\Asubdivision w \cA$. Here, a face of $\Asubdivision w \cA$ is an \emph{interior face} if it is not entirely contained in the boundary of $\conv \cA$. In particular, the structure of $\tightspan w \cA$ determines the structure of $\Asubdivision w \cA$ and vice versa.

We now consider splits of point configurations (see \cite{MR2502496} for details on splits of polytopes and \cite{Herrmann09b} for generalisations to point configurations): A \emph{split} $T$ of a point configuration~$\cA$ is a subdivision of $\cA$  which
has exactly two maximal faces denoted by $T_+$ and $T_-$
(see e.g. Figure~\ref{fig:splits}). The affine hull of $T_+\cap T_-$ is a hyperplane $H_T$ (in the affine hull of $\cA$), the \emph{split hyperplane} of $T$ with respect to~$\cA$. Conversely, it follows from (SD2) and (SD3) that a hyperplane defines a split of $\cA$ if and only if its intersection with the (relative) interior of~$\cA$ is nontrivial and it does not separate the endpoints of any edge of $\cA$. A split $T$ is a regular subdivision, so we have a lifting function $w_T$ such that $\subdivision{w_T}{\cA}=T$; see \cite[Lemma~3.5]{MR2502496}. A set $\cT$ of splits of $\cA$ is called \emph{compatible} if for all $T_1,T_2\in \cT$ the intersection of $H_{T_1}\cap H_{T_2}$ with the relative interior of $\conv \cA$ is empty.

\begin{figure}
    \includegraphics[width=.6\textwidth]{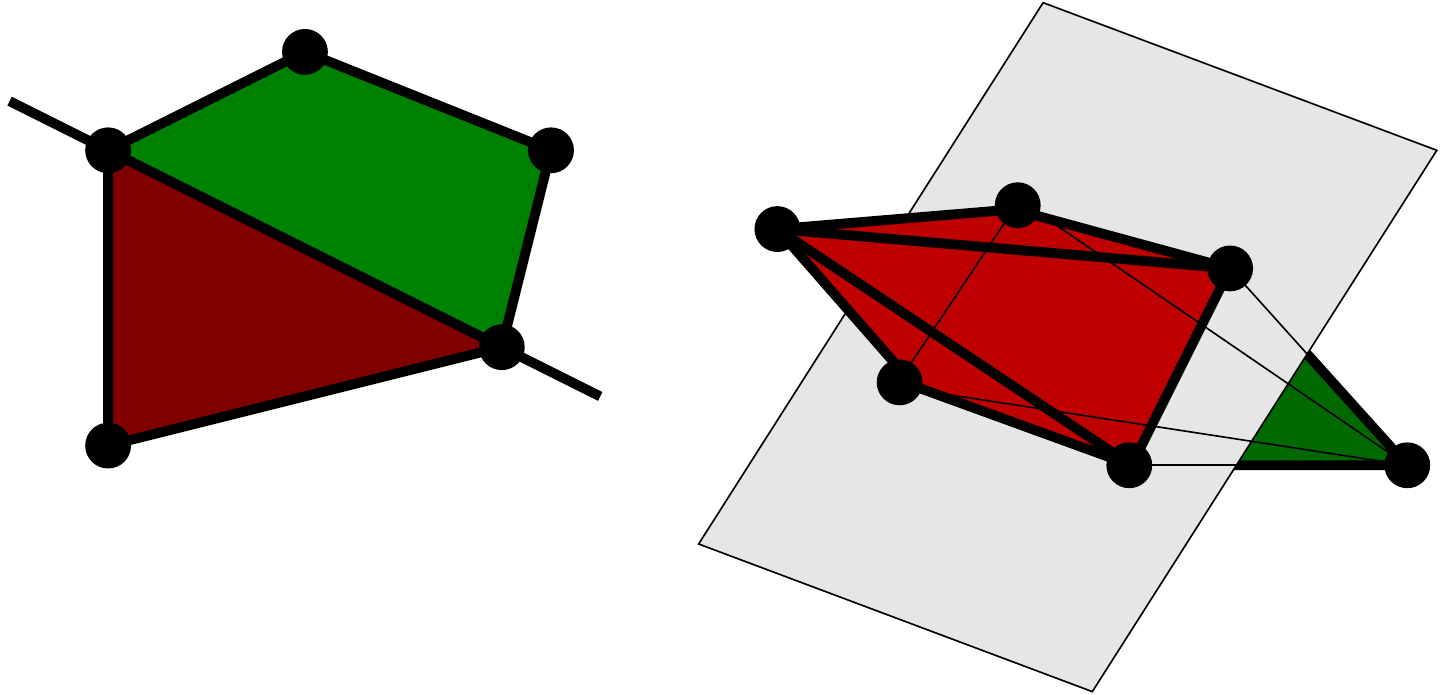}
\caption{To the left a split $T$ of a configuration of five points forming a pentagon together with its split hyperplane (line) $H_T$. To the right a split of the point configuration $\bar \cA(X)$ for $\card X=4$ whose convex hull is an octahedron.}
\label{fig:splits}
\end{figure}

The following observation, which is a slight generalisation of~\cite[Proposition 4.6]{MR2502496}, characterises when the tight-span of a point configuration is a tree and will be the key to some of our results.

\begin{prop}\label{prop:1-dim:comp}
Let $\cA$ be a point configuration and $w:\cA \to \RR$ a weight function. Then the tight-span $\tightspan w \cA$ is a tree if and only if the subdivision $\subdivision w \cA$ is a common refinement of compatible splits of $\cA$.
\end{prop}

An important theorem concerning splits of point configurations is the Split Decomposition Theorem; see \cite[Theorem~3.10]{MR2502496} and \cite[Theorem~2.2]{MR2252108}. It states that each weight function $w$ inducing a subdivision $\subdivision{w}{\cA}$ of a point configuration $\cA$ can be uniquely decomposed in a certain coherent way into a split prime weight function and a sum of split weight functions. Here, we only need the following direct corollary of this fact:

\begin{cor}\label{cor:split-decomposition}
Let $\cA$ be a point configuration and $w:\cA\to\RR$ a weight function such that $\subdivision w \cA$ is a common refinement of a set $\cT$ of compatible splits of $\cA$. Then there exists a function $\alpha:\cT\to\RR_{>0}$ such that
\[
w=\sum_{T\in\cT}\alpha(T)w_T\,.
\]
\end{cor}

\section{Splits of Sets and Point Configurations}\label{sec:splits}

When relating the tight-span of metrics on $X$ and the point configuration $\bar \cA(X)$ (the set of vertices of the second hypersimplex $\Hypersimplex 2X$), a key observation of \cite{MR2252108} is that a split of the set $X$ corresponds to a split of the point configuration $\bar \cA(X)$ and vice versa. We now explain how this fact leads to some further relationships between splits of $X$ and splits of $\bar\cA(X)$.

Let $X$ be a finite set and $A,B\subseteq X$ two non-empty subsets with $A\cap B=\emptyset$. The collection $\{A,B\}$ is called a \emph{partial split} of $X$. The pair $(A,B)$ is called a \emph{directed partial split} of $X$. If in addition $A\cup B=X$, we call $\{A,B\}$ a \emph{split} of $X$ and $(A,B)$ a \emph{directed split} of $X$.
For a subset $C\subseteq X$ a (partial) split $\{A,B\}$ of $X$ is said to \emph{split}~$C$ if neither of the intersections $A\cap C$ or $B\cap C$ is empty.

Two partial splits $\{A,B\}$, $\{C,D\}$ of $X$ are called \emph{compatible} if one of the following four conditions is satisfied:
\begin{align}\label{cond:split:compatible}
A&\subseteq C \quad\text{and}\quad B\supseteq D\,,\notag\\
A&\subseteq D\quad \text{and}\quad B\supseteq C\,,\\
A&\supseteq C \quad\text{and} \quad B\subseteq D\,,\notag\\
A&\supseteq D \quad\text{and} \quad B\subseteq C\,.\notag
\end{align}
Note that this condition implies that there exist $E\in\{A,B\}$ and $F\in\{C,D\}$ such that $E\cap F=\emptyset$, a characterisation usually taken as the definition of compatibility for splits of $X$ (see, e.g., \cite{MR2060009}).
Two directed partial splits $(A,B)$, $(C,D)$ of~$X$ are called \emph{compatible} if one of the following four conditions is satisfied:
\begin{align}\label{cond:dir-split:compatible}
A&\subseteq C \quad\text{and}\quad B\supseteq D\,,\notag\\
A&\subseteq X\setminus C \quad \text{and}\quad B\supseteq X\setminus D\,,\\
A&\supseteq C \quad\text{and} \quad B\subseteq D\,,\notag\\
X\setminus A&\supseteq C \quad\text{and} \quad X\setminus B\subseteq  D\,, \notag
\end{align}

Note that Condition~\eqref{cond:dir-split:compatible} implies Condition~\eqref{cond:split:compatible} and that Conditions~\eqref{cond:split:compatible} and~\eqref{cond:dir-split:compatible} are equivalent for (directed) splits. A set $\cS$ of partial (directed) splits of $X$ is called \emph{compatible} if each two elements of $\cS$ are compatible. Furthermore, a set $\cS$ of directed splits of $X$ is called \emph{strongly} compatible if there exists an ordering $(A_1,B_1),\dots,(A_l,B_l)$ of the elements of $\cS$ such that $A_i\subseteq A_{i+1}$ and $B_i\supseteq B_{i+1}$ for all $1\leq i<l$.

We now investigate the relation of these different kinds of compatibility with compatibility of splits of the point configurations defined in the introduction.

First, we consider the point configuration $\cA(X)$. Splits of this point configuration were first studied by Hirai~\cite{MR2233884,MR2252108}. 

\begin{prop}[Proposition~4.4 in \cite{MR2252108}]\label{prop:splits-a}
Let $X$ be a finite set. For a partial split $\{A,B\}$ of $X$ the hyperplane given by the equation
\[
\sum_{i\in A} f(i)=\sum_{i\in B} f(i)
\]
defines a split of the point configuration $\cA(X)$. Moreover, all splits of $\cA(X)$ arise in this way.
\end{prop}

The compatibility can be characterised as follows:
\begin{prop}[Theorem~2.3 in \cite{MR2233884}]\label{prop:partial-splits-comp}
A set $\cS$ of partial splits of $X$ is compatible if and only if $\smallSetOf{T_S}{S\in\cS}$ is a compatible set of splits of $\cA(X)$.
\end{prop}

Now we consider the point configurations  $\bar \cB(X,Y)$, first describing what their splits are.

\begin{prop}
Let $X$ and $Y$ be two disjoint finite sets with $\card X,\card Y\geq 2$ and $A\subsetneq X$, $B\subsetneq Y$ non-empty. Then the hyperplane given by
\begin{align}\label{eq:p_o_s-splits}
\sum_{i\in A}f(i)=\sum_{j\in B}f(j)
\end{align}
defines a split of the point configuration $\bar \cB(X,Y)$. Moreover, all splits of $\bar \cB(X,Y)$ arise in this way.
\end{prop}

Note that taking the complements of $A$ and $B$ simultaneously yields the same split of $\bar \cB(X,Y)$ (but there are no other choices).

\begin{proof}
First we remark that for all non-empty $A\subsetneq X$, $B\subsetneq Y$ the function $f\in \RR^{X\cup Y}$ defined by \begin{align*}
f(i)\ =\ \begin{cases} \frac1{2|A|}, & \text{if }i\in A\,,\\ \frac1{2\card{X\setminus A}}, & \text{if }i\in X\setminus A\,,\\ \frac1{2|B|}, & \text{if }i\in B\,,\\ \frac1{2\card{ Y\setminus B}}, & \text{if }i\in Y\setminus B\,,\end{cases}
\end{align*}
 is in the interior of $\conv \bar \cB(X,Y)$ since $0<f(i)<1$ for all $i\in X\cup Y$ and $\sum_{i\in X\cup Y}f(i)=2$. It is also in the hyperplane defined by Equation~\eqref{eq:p_o_s-splits}, since $\sum_{i\in A}f(i)=1$, $\sum_{i\in{B}} f(i)=1$. Hence all those hyperplanes meet the interior of $\conv \bar \cB(X,Y)$.
 
By the definition of $\bar \cB(X,Y)$, two vertices $u=(f_1,f_2)$, $v=(g_1,g_2)\in\RR^X\times\RR^Y$ of $\bar \cB(X,Y)$ are connected by an edge if and only if $f_1=g_1$ or $f_2=g_2$. So by going from $u$ to $v$ along an edge, the value on at most one side of Equation~\eqref{eq:p_o_s-splits} changes by at most $1$. Since for all elements of $\bar \cB(X,Y)$ all values occurring in Equation~\eqref{eq:p_o_s-splits} are integers, the corresponding hyperplane does not cut an edge of $\bar\cB(X,Y)$ and hence defines a split of $\bar\cB(X,Y)$.

Now, let $H=\{f\in\RR^{X\cup Y}\st\sum_{i\in X\cup Y} \alpha_i f(i)=0\}$ define a split of $\bar \cB(X,Y)$ for some $\alpha_i\in\RR$. We can assume that the first non-zero $\alpha_i$ is equal to $1$. However, since the matrix of vertices of a product of simplices is totally unimodular (i.e., all the determinants of all square submatrices are in $\{0,1,-1\}$ \cite{MR976522}), all other non-zero~$\alpha_j$ have to be equal to $\pm1$. Also, note that the product of simplices $\conv \bar \cB(X,Y)$ has $\card Y$ facets that are isomorphic to $\card X$"=dimensional simplices. The hyperplane $H$ has to meet at least one of these facets non-trivially. Since simplices have no splits, $H$ has to define a face of this facet~$F$. So we can conclude that we cannot have $\alpha_i=-\alpha_j$ for $i,j\in X$ since $H$ would then meet the interior of $F$. The only remaining possibility for $H$ is therefore Equation~\eqref{eq:p_o_s-splits} for arbitrary $A$ and $B$. Since for $A=\emptyset,$ $B=\emptyset$, $A=X$, or $B=Y$ this hyperplane would not intersect the relative interior of $\conv \bar \cB(X,Y)$, the proof is complete.
\end{proof}

Thus, a split of the point configuration $\bar \cB(X,Y)$ is defined by two sets $\emptyset \not=A\subsetneq X$ and $\emptyset \not=B\subsetneq Y$ and this representation is unique up to simultaneously taking the complements of $A$ and~$B$. Compatibility can now be characterised as follows

\begin{prop}\label{prop:p_o_s-compatible}
Let $T$ be a split of $\bar \cB(X,Y)$ defined by $A\subsetneq X$ and $B\subsetneq Y$, and $U$ a split of $\bar \cB(X,Y)$ defined by $C\subsetneq X$ and $D\subsetneq Y$. Then $T$ and $U$ are compatible if and only if one of the following conditions is satisfied:
\begin{align}\label{eq:p_o_s-compatible}
A&\subseteq C \quad\text{and}\quad B\supseteq D\,,\notag\\
A&\subseteq X\setminus C \quad \text{and}\quad B\supseteq Y\setminus D\,,\\
A&\supseteq C \quad\text{and} \quad B\subseteq D\,,\text{ or}\notag\\
X\setminus A&\supseteq C \quad\text{and} \quad Y\setminus B\subseteq D\,.\notag
\end{align}
\end{prop}

\begin{proof}
We define the sets
\begin{align*}
A_1&=A \setminus C,& A_2&=C \setminus A,& A_3&=A \cut C,& A_4&=X\setminus (A\union C),\\
B_1&=B \setminus D,& B_2&=D \setminus B,& B_3&= B \cut D,& B_4&=Y\setminus (B\union D)\,;
\end{align*}
and set
$$X_i\ =\ \sum_{i\in A_i}f(i),\quad Y_i\ =\ \sum_{i\in B_i}f(i)\,.$$


Then the hyperplanes for $T$ and $U$ are defined by
\begin{align}\label{eq:p_o_s-comp-splits}
X_1+X_3\ =\ Y_1+Y_3\quad \text{and}\quad X_2+X_3\ =\ Y_2+Y_3\,,
\end{align}
respectively. If we subtract these two equations, we get
\begin{align}\label{eq:p_o_s-comp-splits2}X_1-X_2\ =\ Y_1-Y_2\,.\end{align}

We first prove that Condition \eqref{eq:p_o_s-compatible} is sufficient for compatibility of $T$ and $U$. So suppose $T$ and $U$ are not compatible, $A\subseteq C$, and $D \subseteq B$.  This implies that there exists some $x\in H_T\cut H_U$ in the interior of $\conv \bar \cB(X,Y)$ with $X_1=Y_2=0$. From this, Equation~\eqref{eq:p_o_s-comp-splits2} and the fact that $\bar \cB(X,Y) \subseteq \RR_{\geq0}^{X\cup Y}$, we can further conclude that $X_2=Y_1=0$. So $x_i=0$ for all $i\in A_2 \union B_1$. However, if $x_i=0$ for some $i\in X\cup Y$, the point $x$ would be contained in the boundary facet of $\conv \bar \cB(X,Y)$ defined by $x_i=0$. So $A_2$ and $B_1$ are empty, and hence $T= U$. The second case follows similarly.

For the necessity, assume that \eqref{eq:p_o_s-compatible} does not hold. This is equivalent to
\begin{align}\label{eq:p_o_s-cond}
  A_1\not=\emptyset\text{ or }B_2\not=\emptyset,A_3\not=\emptyset\text{ or }B_4\not=\emptyset, A_2\not=\emptyset\text{ or }B_1\not=\emptyset,\text{ and }A_4\not=\emptyset\text{ or }B_3\not=\emptyset\,.
\end{align}

We will now distinguish several cases, depending on the number of sets $A_j$, $B_j$ which are empty. In each case we will give a point $x\in\relint(\conv\bar\cB(X,Y))\cut H_S \cut H_T$ . This will be done by assigning values in the interval $(0,1)$ to all $X_j,Y_j$ for which $A_j,B_j$, respectively, are non-empty such that Equation~\eqref{eq:p_o_s-comp-splits2} holds and $\sum X_j=\sum Y_j =1$. The explicit values of~$f\in\RR^{X\cup Y}$ are then obtained by setting $f(i)=\frac {X_j}{\card {A_j}}$, $f(i)=\frac{Y_j}{\card {B_j}}$, for $i\in A_j$, $i\in B_j$, respectively.

\noindent \textbf{Case 1:} None of the sets $A_j, B_j$ is empty. Then we simply set $X_j,Y_j=\frac14$ for all $j\in\{1,2,3,4\}$.

\noindent \textbf{Case 2:} One of the sets $A_j, B_j$ is empty. We assume without loss of generality that~$A_1=\emptyset$. Then we set $X_3=\frac12$, $X_4=Y_2=\frac38$, $Y_1=Y_3=\frac 14$, and $X_2=Y_4=\frac18$.

\noindent \textbf{Case 3:} Two of the sets $A_j, B_j$ are empty. As in Case~2, we assume that one of these sets is $A_1$. Using \eqref{eq:p_o_s-cond}, and taking into account that  neither $A,B,C,D$ nor their complements (in $X$ and $Y$, respectively) can be empty, we get the following possibilities:
\begin{itemize}
\item $A_1=A_2=\emptyset$: Set $X_3=X_4=\frac 12$ and $Y_i=\frac14$ for all $i\in\{1,2,3,4\}$.
\item $A_1=B_1=\emptyset$: Set $X_3=Y_3=\frac 12$ and $X_2=X_4=Y_2=Y_4=\frac14$.
\item $A_1=B_3=\emptyset$: Set $X_4=Y_2=\frac 12$ and $X_2=X_3=Y_1=Y_4=\frac14$.
\item $A_1=B_4=\emptyset$: Set $X_3=Y_2=\frac 12$ and $X_2=X_4=Y_1=Y_3=\frac14$.
\end{itemize}

\noindent \textbf{Case 4:} Three of the sets $A_j, B_j$ are empty. We again assume that $A_1$ is one of the sets. There remain three possibilities:
\begin{itemize}
\item $A_1=A_2=B_3=\emptyset$: Set $X_4=\frac 23$ and $X_3=Y_1=Y_2=Y_4=\frac13$.
\item $A_1=A_2=B_4=\emptyset$: Set $X_3=\frac 23$ and $X_4=Y_1=Y_2=Y_3=\frac13$.
\item $A_1=B_3=B_4=\emptyset$: Set $Y_2=\frac 23$ and $X_2=X_3=X_4=Y_1=\frac13$.
\end{itemize}

\noindent \textbf{Case 5:} Four of the sets $A_j, B_j$ are empty. By assuming that $A_1$ is one of them, this yields $A_1=A_2=B_3=B_4=\emptyset$. Set $X_3=X_4=Y_1=Y_2=\frac 12$.

\end{proof}

Note that in the special case where $Y$ is a disjoint copy of $X$ a directed partial split $S=(A,B)$ of $X$ gives rise to a split $T_S$ of $\bar\cB(X)$, namely the one defined by $A$ and $B$. In general, not all splits of $\bar \cB(X)$ arise in this way, since we assume that $A$ and $B$ are disjoint. Proposition~\ref{prop:p_o_s-compatible} then gives us the following.

\begin{cor}\label{cor:splits-axx:compatible}
A collection $\cS$ of directed partial splits of $X$ is compatible if and only if $\smallSetOf{T_S}{S\in\cS}$ is a compatible system of splits for $\bar \cB(X)$.
\end{cor}

Note that a characterisation of the splits of the point configuration $\cC(X)$ (the vertices of the cube) and their compatibility is given in \cite[Propositions~3.15 and 3.16]{H_thesis}; we refrain from stating or proving it here since we will not use it later.

\section{Tight-Spans of Symmetric Maps, Distances and Metrics}\label{sec:symmetric-functions}

A symmetric map $D:X\times X \to \RR$ with $D(x,x)=0$ for all $x\in X$ and $D(x,y)\geq 0$ for all $x,y\in X$ is called a \emph{distance} (or dissimilarity map) on $X$. It is called a \emph{metric} on $X$ if it additionally satisfies $D(x,y)+D(y,z)\geq D(x,y)$ for all $x,y,z\in X$ (\emph{triangle inequality}).

Now, consider the vector space $\RR^X=\{f:X\to \RR\}$ of all functions $X\to \RR$ with the natural basis $\smallSetOf{e_x}{x\in X}$ where $e_x\in \RR^{X}$ denotes the function sending $x$ to $1$ and all other elements of $X$ to $0$. Then the tight-span $T_D$ of a symmetric map $D: X\times X \to \RR$ is defined to be the set of all minimal elements of the polyhedron
\[
P_D=\SetOf{f\in\RR^X}{f(x)+f(y)\geq D(x,y) \text{ for all } x,y\in X}\,.
\]

Note that if $D$ is a distance, this is exactly the definition of $P_D$ and $T_D$ given by Hirai~\cite[Section 2.3]{MR2233884}. Also, if $D$ is a metric, $T_D$ corresponds to the tight-span of the metric $D$ as defined by Isbell \cite{MR82949} and Dress \cite{MR753872}.

Now, given a symmetric map $D: X\times X \to \RR $ we define a weight function $w^D:\cA(X)\to \RR$ on the point configuration $\cA(X)$ via $w^D(e_x+e_y)=-D(x,y)$. The following proposition is the key to deriving the relation between tight-spans of symmetric functions and tight-spans of the point configuration $\cA(X)$. In the special case where $D$ is a metric, this was the observation of Sturmfels and Yu~\cite{MR2097310} mentioned in the introduction.

\begin{prop}\label{prop:ts-dissimilarity}
Let $D: X\times X \to \RR$ be a symmetric function. Then we have:
\begin{enumerate}
\item $P_D=\envelope{w^D}{\cA(X)}$, and
\item $T_D=\tightspan{w^D}{\cA(X)}$.
\end{enumerate}
\end{prop}
\begin{proof}
\begin{enumerate}
\item We have 
\begin{align*}
P_D&=\SetOf{f: X \to \RR}{f(x)+f(y)\geq D(x,y) \text{ for all } x,y\in X}\\
&=\SetOf{f: X \to \RR}{\scp {e_x+e_y}f \geq D(x,y)\text{ for all } x,y\in X}\\
 &=\SetOf{f: X \to \RR}{\scp af\geq -w^D(a) \text{ for all } a\in\cA(X)}\\
&=\envelope{w^D}{\cA(X)}\,.
\end{align*}
\item Obviously, $\cA(X)$ is positive and $\envelope{w^D}{\cA(X)}$ is bounded from below since we have the inequalities $2e_x\geq -D(x,x)$ for all $x\in X$. Thus, the claim follows  from (a) and Lemma~\ref{lem:minimal}.
\end{enumerate}
\end{proof}

 We will now see that tight-spans of general symmetric maps and tight-spans of distances are essentially the same in the sense that they only differ by a simple shift.

\begin{prop}\label{prop:ts-sym-div}
Let $D:X\times X \to \RR$ be a symmetric function, $D'$ defined via
\[
D'(x,y)=D(x,y)-\frac 12\left( D(x,x)+D(y,y)\right)\,,
\]
and $ v:X\to \RR$ defined by $v(x)=\frac12 D(x,x)$.
Then $T_D=T_{D'}+v$.
\end{prop}
\begin{proof}
Our definitions imply that for $a=e_x+e_y\in \cA(X)$, we have
\[ w^{D'}(a)=-D(x,y)+\frac 12\left( D(x,x)+D(y,y)\right)=w^D(a)+\scp av\,. \]
Hence the claim follows from Lemma~\ref{lem:ts-shift}.
\end{proof}

Obviously, for an arbitrary symmetric map $D: X\times X\to \RR$ we have $D'(x,x)=0$ for all $x\in X$. However, $D'$ is not necessarily a distance, since $D'$ need not to be positive, even if $D$ is. Even so, the following lemma shows that the negative values of $D$ can be ignored when looking at the tight-span:

\begin{lem}\label{lem:ts-sym-div}
Let $D: X\times Y\to \RR$ be a symmetric function with $D(x,x)=0$ for all $x\in X$, and $D_+$ be defined by $D_+(x,y)=\max(0,D(x,y))$ for all $x,y\in X$. Then $T_D=T_{D_+}$.
\end{lem}
\begin{proof}
For two symmetric maps $E,F: X \to \RR$ with $E\geq F$ (pointwise) one directly sees that $P_E\subseteq P_F$, so  $P_{D_+}\subseteq P_D$. On the other hand, for all $f\in P_D$ and $x\in X$, we have $2f(x)\geq D(x,x)=0$, hence $f(x)+f(y)\geq 0\geq \max(0,D(x,y))=D_+(x,y)$ for all $x,y\in X$. This implies $P_D\subseteq P_{D_+}$.
Altogether, we have $P_D=P_{D_+}$ and hence $T_D=T_{D_+}$.
\end{proof}

So, by Proposition~\ref{prop:ts-sym-div} and Lemma~\ref{lem:ts-sym-div}, we get:

\begin{cor}\label{cor:ts-sym-div}
For any symmetric map $D:X\times X \to \RR$ there exists a distance map (namely $(D')_+:X\times X \to \RR$) and some $ v\in \RR^X$ such that $T_D=T_{(D')_+}+v$.
\end{cor}

We now turn to trees: Given a tree $T=(V,E)$, an edge-length function $\alpha:E\to \RR$ and a family $\cF=\smallSetOf{F_x}{x\in X}$ of subtrees of $T$, we can define a distance $D$ on $X$ by setting $D(x,y)=\min\smallSetOf{D_T(x,y)}{x\in F_x,y\in F_y}$. A distance arising in that way is called a \emph{distance between subtrees of a tree}.
Those distances can also be characterised in terms of partial splits: Given a partial split $P=\{A,B\}$ of $X$, a corresponding distance on $X$ is the defined by
\begin{equation}\label{eq:partial-distance}
d_P(i,j)=\begin{cases} 1,&\text{if } i\in A, j\in B \text{ or } i\in B, j\in A,\\
0,&\text{else\,,}\end{cases}
\end{equation}
for all $i,j\in X$. Furthermore, for a set $\cP$ of partial splits of $X$ and a function $\alpha\to\RR_{>0}$, we define $d_{(\cP,\alpha)}=\sum_{P\in \cP}\alpha(P)d_P$.

Using the relation between splits of the point configuration $\cA(X)$ and partial splits of $X$(Propositions~\ref{prop:splits-a} and \ref{prop:partial-splits-comp}), we get the following result, that can  alternatively be inferred from Hirai~\cite[Theorem~2.3]{MR2233884} together with Proposition~\ref{prop:ts-dissimilarity}.

\begin{thm}\label{thm:distance-tree}
Let $D$ be a distance. Then the following are equivalent:
\begin{enumerate}
\item The tight-span $\tightspan {w^D}{\cA(X)}$ is a tree.
\item $D$ is a distance of weights between subtrees of a tree.
\item There exists a compatible set $\cP$ of partial splits of $X$ and $\alpha:\cP\to \RR_{>0}$ such that $D=d_{(\cP,\alpha)}$.
\end{enumerate}
\end{thm}


\section{Tight-Spans of Non"=Symmetric Maps}\label{sec:non-symmetric}

Hirai and Koichi~\cite{HiraiKoichi} introduced the concept of tight-spans of directed distances in order to study certain multicommodity flow problems. In this section, we will show that these tight-spans can also be considered as tight-spans of point configurations and we will generalise their concept to general non-symmetric maps.

Let $X$ and $Y$ be finite sets with $X\cap Y=\emptyset$ and $D:X\times Y\to \RR$ an (arbitrary) map. We define the polyhedra
\begin{align*}
\Pi_D&=\SetOf{f:X\cup Y \to \RR}{f(x)+f(y)\geq D(x,y) \text{ for all } x\in X, y\in Y}, \text{ and }\\
 P_D&=\Pi_D\cap \RR^{X\cup Y}_{\geq 0}\,.
\end{align*}
In addition, the sets of minimal elements of $\Pi_D$ and $P_D$ are called $\Theta_D$ and $T_D$, respectively. The set $T_D$ is called the \emph{tight-span} of $D$.

Recall that $\bar \cB(X,Y)\subseteq \RR^{X\cup Y}$ is defined as the configuration of all points $e_x+e_y$ with $x\in X$, $y\in Y$ and also that $\cB(X,Y)=\bar \cB(X,Y)\cup \smallSetOf{2 e_x}{x\in X\cup Y}$. Note that $\cB(X,Y),\bar\cB(X,Y)\subseteq \cB(X\cup Y)$ and that $\conv \cB(X,Y)$ is the simplex $\conv\smallSetOf{2 e_x}{x\in X\cup Y}$, whereas all elements of $\bar\cB(X,Y)$ are vertices of $\conv \bar \cB(X,Y)$, which is the product of a $(\card X-1)$"= and a $(\card Y-1)$"=dimensional simplex.

To the map $D: X\times Y \to \RR $ we associate weight functions $\bar w^D:\bar \cB(X,Y)\to \RR$ by $\bar w^D(e_x+e_y)=-D(x,y),x\in X,y\in Y$ and $w^D:\cB(X,Y)\to \RR$ by
\[
w^D(a)=\begin{cases}
 \bar w^D(a),&\text{if } a=e_x+e_y,x\in X,y\in Y,\\
0,&\text{else}\,.
\end{cases}
\]
The following can be proven in a similar way to Proposition~\ref{prop:ts-dissimilarity}:

\begin{prop}\label{prop:ts-directed}
Let $D: X\times Y \to \RR$, then we have
\begin{enumerate}
\item $\Pi_D=\envelope{\bar w^D}{\bar\cB(X,Y)}$,
\item $\Theta_D=\tightspan{\bar w^D}{\bar\cB(X,Y)}$,
\item $P_D=\envelope{w^D}{\cB(X,Y)}$, and
\item $T_D=\tightspan{w^D}{\cB(X,Y)}$.
\end{enumerate}
\end{prop}

It was shown in \cite[Lemma~22]{MR2054977} that $\Theta_D$ is piecewise"=linear isomorphic to the tropical polytope (see, e.g., Develin and Sturmfels \cite{MR2054977}) with vertex set $\smallSetOf{(D(x,y))_{y\in Y}}{x\in X}$. Similarly, using  a proof like that for Lemma~\ref{lem:ts-sym-div}, we have:

\begin{lem}\label{lem:ts-map-div}
Let $D: X\times Y\to \RR$ and $D_+$ defined by $D_+(x,y)=\max(0,D(x,y))$ for all $x\in X$, $y\in Y$. Then $T_D=T_{D_+}$.
\end{lem}

Of particular interest is the case where $Y$ is a disjoint copy of $X$, that is, $D$ is a (not necessarily symmetric) function from $X\times X$ to $\RR$. We denote the two distinct copies of $X$ by $X_l$ and $X_r$ and set $X_d=X_l\cup X_r$. If in this case $D\geq 0$ and $D(x,x)=0$ for all $x\in X$ then $D$ is called a \emph{directed distance}; and if $D$ also satisfies the triangle inequality, it is called a \emph{directed metric}. These were considered by Hirai and Koichi~\cite{HiraiKoichi}. In fact, our definitions of $\Pi_d$, $\Theta_D$, $P_D$ and $T_D$ are generalisations of their definitions.

As in the case of symmetric maps, it can be deduced from Lemma~\ref{lem:ts-shift} that a map $D:X\times X \to \RR$ can be transformed  into a map with $D(x,x)=0$ for all $x\in X$ by shifting it by a vector $v\in\RR^{X_d}$ defined by $v(x)=1/2 D(x,x)$ for $x\in X_d$ (here $x$ is identified with its copy in $X$). Together with Lemma~\ref{lem:ts-map-div} this leads to the following corollary.

\begin{cor}\label{cor:ts-map-div}
For each map $D:X\times X \to \RR$ there exists a directed distance $D'$ on $X$ and some $ v\in \RR^{X_d}$ such that $T_D=T_{D'}+v$.
\end{cor}

To a directed distance $D$ on $X$ we associate a (symmetric) distance $D^u :X_d\times X_d \to \RR_{\geq 0}$ by setting
\[
D^u(x,y)=\begin{cases} D(x,y),&\text{if }x\in X_l \text{ and } y\in X_r\text{ or }x\in X_r \text{ and } y\in X_l\,,\\
              0,&\text{else}\,.\end{cases}
\]

It follows that $P_{D^u}=P_{D}$ and $T_{D^u}=T_{D}$. So tight-spans of undirected distances are just tight-spans of special directed distances.


In \cite{HiraiKoichi}, Hirai and Koichi give explicit conditions on $D$ for when $T_D$ and $\Theta_D$ are trees. We now give new and we feel somewhat conceptually simpler proofs of these results using point configurations.

We begin by recalling some basic definitions from \cite[Section~3]{HiraiKoichi}. An \emph{oriented tree} $\Gamma=(V(\Gamma),E(\Gamma))$ is a directed graph whose underlying undirected graph is a tree. For an oriented tree $\Gamma$ and an edge-length function $\alpha:E(\Gamma)\to \RR_{\geq0}$, we define a directed distance $D_{\Gamma,\alpha}$ on $V(\Gamma)$ by setting $D_{\Gamma,\alpha}(x,y)=\sum_{e\in \arrow P(x,y)}\alpha(e)$ for all $x,y\in V(\Gamma)$, where $\arrow P(x,y)$ is the set of all edges on the unique (undirected) path from $x$ to $y$ that are directed from $x$ to $y$. For $A,B\subseteq V(\Gamma)$ we set $D_{\Gamma,\alpha}(A,B)=\min\smallSetOf{D_{\Gamma,\alpha}(a,b)}{a\in A,b\in B}$. For an undirected distance $D:X\times X\to \RR$ and a family $\cF=\smallSetOf{F_x}{x\in X}$ of subtrees of $\Gamma$, we say that $(\Gamma,\alpha,\cF)$ is an \emph{oriented tree realisation} of $D$ if
\[
D(x,y)=D_{\Gamma,\alpha}(F_x,F_y)\text{ for all }x,y\in X\,.
\]

Now, let $S=(A,B)$ be a directed partial split of $X$. We can define a directed distance $D_S$ on $X$ by setting
\[
D_S(x,y)=\begin{cases}
1,&\text{if $x\in A$ and $y\in B$,}\\
0,&\text{else.}
\end{cases}
\]
Note that $D_S$ is a directed metric if and only if $S$ is a directed split. Also note that the oriented tree $\Gamma$ consisting of a single edge $(v,w)$ with weight $\lambda$ and
\[
F_x=\begin{cases}
\{v\}&\text{if }x\in A,\\
\{w\}&\text{if }x\in B,\\
\Gamma&\text{else},
\end{cases}
\]
for all $x\in X$, is an oriented tree realisation of $D_S$.

Now let $D$ be an arbitrary directed distance with oriented tree realisation $(\Gamma,\alpha,\cF)$. For each $e=(a,b)\in E(\Gamma)$ we define a directed partial split $S_e=(A_e,B_e)$ of $X$, where $A_e$ is the set of all $x \in X$ whose subtree $F_x\in \cF$ is entirely contained in the same connected component of $\Gamma\setminus e$ as $a$ and $B_e$ the same for $b$. (Here $\Gamma\setminus e$ denotes the forest obtained from $\Gamma$ by deleting the edge $e$.) It now follows that $D=\sum_{e\in E(\Gamma)} \alpha(e) D_{S_e}$. 

We can now show the following:
\begin{prop}\label{prop:oriented-tree-realisation}
Let $\cS$ be a set of directed partial splits, $\alpha:\cS\to \RR_{>0}$ a function, and $D$ the directed distance defined by
\[
D=\sum_{S\in\cS} \alpha(S) D_{S}\,.
\]
Then we have:
\begin{enumerate}
\item \label{prop:oriented-tree-realisation:comp} $D$ has an oriented tree realisation $(\Gamma,\alpha,\cF)$ where each $F\in \cF$ is a directed path if and only if $\cS$ is compatible.
\item\label{prop:oriented-tree-realisation:strong-comp} $D$ has an oriented tree realisation $(\Gamma,\alpha,\cF)$ such that $\Gamma$ is a directed path if and only if $\cS$ is strongly compatible.
\end{enumerate}
\end{prop}
\begin{proof}
\begin{enumerate}
\item Let $e$ and $f$ be two edges of $\Gamma$ and  $S_e$ and $S_f$ the associated  directed partial splits. We have to show that $S_e$ and $S_f$ are compatible. Suppose first that in any undirected path in $\Gamma$ containing $e$ and $f$ these two edges are directed in the same direction. (Since $\Gamma$ is a tree, this is the case if there exists some undirected path with this property.) Then (possibly after exchanging $e$ and $f$) we have $A_e\subseteq A_f$ and $B_f\subseteq B_e$, which implies that $S_e$ and $S_f$ are compatible by the first condition of \eqref{cond:dir-split:compatible}.

Now, suppose that in any undirected path in $\Gamma$ containing $e$ and $f$ these two edges are directed in the different directions. Because each $F\in \cF$ is a directed path, this implies that there does not exist an $F\in \cF$ that contains $e$ as well as $f$. Hence, we get (again after a possible exchange of $e$ and $f$) $X\setminus A_e\subseteq A_f$ and $B_e\subseteq X\setminus B_f$, implying that $S_e$ and $S_f$ are compatible by the third condition of \eqref{cond:dir-split:compatible}. On the other hand, given a compatible set of directed partial splits of $X$, we can construct a corresponding tree $\Gamma$ with a family of subtrees $\cF$ that are directed paths such that  $(\Gamma,\alpha,\cF)$ is an oriented tree realisation of $D$.
\item Follows directly from the definition.
\end{enumerate}
\end{proof}

Now we can state and prove the tree-like theorem.

\begin{thm}[Theorems~3.1 and~3.2 in~\cite{HiraiKoichi}]\label{thm:tight-span:tree}

Let $D$ be a directed distance on $X$. Then we have:
\begin{enumerate}
\item $\Theta_D$ is a tree if and only if $D$ has an oriented tree realisation $(\Gamma,\alpha,\cF)$ such that each $F\in \cF$ is a directed path.\label{thm:tight-span:tree:theta}
\item $T_D$ is a tree if and only if $D$ has an oriented tree realisation $(\Gamma,\alpha,\cF)$ such that $\Gamma$ is a directed path.
\end{enumerate}
\end{thm}

\begin{proof}
\begin{enumerate}
\item
Let $D$ be a directed distance on $X$ such that $D$ has on oriented tree realisation $(\Gamma,\alpha,\cF)$ such that each $F\in \cF$ is a directed path. By Proposition~\ref{prop:oriented-tree-realisation} \eqref{prop:oriented-tree-realisation:comp}, this implies that there exists a set $\cS$ of directed partial splits of $X$ such that
\[
D=\sum_{S\in\cS} \alpha(S) D_{S}\,.
\]
For each $S\in \cS$ and $a\in \bar \cB(X)$ we have 
\[
\bar w^{D_S}(a)=\begin{cases}
-1,&\text{if $a=e_x+e_y$ and $x\in A$ and $y\in B$,}\\
0,&\text{else,}
\end{cases}
\]
and so it is immediately seen that $\bar w^{D_S}$ defines the split $T_S$ of $\bar\cB(X)$. By Corollary~\ref{cor:splits-axx:compatible}, the set $\smallSetOf{T_S}{S\in\cS}$ is a compatible set of splits for $\bar \cB(X)$, and so the subdivision $\subdivision{\bar w^{D}}{\bar\cB(X)}$ is a common refinement of compatible splits. Proposition~\ref{prop:1-dim:comp} now implies that $\Theta_D=\tightspan{\bar w^{D}}{\bar\cB(X)}$ is a tree.

Conversely, let $D$ be a directed distance on $X$ such that $\Theta_D$ is a tree. By Proposition~\ref{prop:1-dim:comp} there exists a compatible set $\cT$ of splits of $\bar\cB(X)$ such that $\subdivision{\bar w^{D}}{\bar\cB(X)}$ is the common refinement of all splits in $\cT$. Since $D$ is a directed distance, we have $w^D(e_x+e_x)=0$ for all $x\in X$, which implies that each $T\in\cT$ is defined by two disjoint subsets of $X$. Hence there exists a directed partial split $S$ with $T=T_S$. Let $\cS$ be the set of all such splits, which is compatible by Corollary~\ref{cor:splits-axx:compatible}. By Corollary~\ref{cor:split-decomposition}, there exists $\alpha_S\in \RR_{\geq 0}$ such that $\bar w^D=\sum_{T\in \cT} \alpha_S w^{D_S}$, so $D=\sum_{\cS\in S} \alpha_S D_S$ and the claim now follows from Proposition~\ref{prop:oriented-tree-realisation} \eqref{prop:oriented-tree-realisation:comp}.
\item The splits of the point configuration $\cB(X)$ are given by partial splits of the set $X_d$. However, given a directed distance $D$, by definition, the value $w^D(e_x+e_y)$ can only be non-zero if $x\in X_l$ and $y\in X_r$. This implies that the only possible partial splits $\{A,B\}$ of $X_d$ that may occur are those with $A\subseteq X_l$ and $B\subseteq X_r$ (or vice versa) and $A\cap B=\emptyset$ (where $A$ and $B$ are considered as subsets of $X$). The splits of this type are in bijection with directed partial splits $(A,B)$ of $X$.

Now, given two such directed partial splits $(A,B)$, $(C,D)$ of $X$, the corresponding splits of $\cB(X)$ are compatible if and only if
\[
 A\subseteq C\text{ and } B\supseteq D,\text{ or }A\supseteq C\text{ and }B\subseteq D\,.
\]
Proposition~\ref{prop:oriented-tree-realisation} \eqref{prop:oriented-tree-realisation:strong-comp} now implies the claim in a similar way as Proposition~\ref{prop:oriented-tree-realisation} \eqref{prop:oriented-tree-realisation:comp} implies Part~\eqref{thm:tight-span:tree:theta}.
\end{enumerate}
\end{proof}

In the case where $D$ is a directed metric, that is, satisfies the triangle inequality, it is obvious that in each directed tree realisation $(\Gamma,\alpha,\cF)$  all $F_x$ are single vertices. In this special case, given a directed tree $\Gamma$, an edge length function $\alpha$ and a map $\phi: X\to V(\Gamma)$, we also call $(\Gamma,\alpha,\phi)$ an \emph{oriented tree realisation} of $D$ if and only if  $(\Gamma,\alpha,\smallSetOf{\phi(x)}{x\in X})$ is an oriented tree realisation of $D$. Note that our approach to oriented tree realisations is somewhat related to those taken by Hakimi and Patrinos \cite{MR0414405} and by Semple and Steel \cite{MR1722236}. However it differs as we take only one 
edge for each undirected edge of a tree rather than two.

\begin{cor}
Let $D$ be a directed metric on $X$. Then we have:
\begin{enumerate}
\item $\Theta_D$ is a tree if and only if $D$ has an oriented tree realisation.
\item $T_D$ is a tree if and only if $D$ has an oriented tree realisation $(\Gamma,\alpha,\phi)$ such that $\Gamma$ is a directed path.
\end{enumerate}
\end{cor}


\section{Diversities as Distances}\label{sec:div-distances}

For a finite set $Y$, we denote the powerset of $Y$ 
by $\cP(Y)$ and let $\ncP(Y)=\cP(Y)\setminus\{\emptyset\}$. A
\emph{diversity} on $Y$ is a 
function $\delta: \cP(Y)\to \RR$ satisfying:
\begin{enumerate}
\item[(D1)] $\delta(A\cup B)+\delta(B\cup C)\geq \delta(A\cup C)$ for all $A,C\in\cP(Y)$ and $B\in\ncP(Y)$, and
\item[(D2)] $\delta(A)=0$ for all $A\in\cP(Y)$ with $\card A\leq 1$.
\end{enumerate}
Diversities were introduced by Bryant and Tupper 
in~\cite{BT10}\footnote{Note that Bryant and Tupper required 
$\delta(A)=0 \iff \card A\leq 1$ in (D2). In particular, 
our maps could be regarded as ``pseudo-diversities'', but for
simplicity we shall just call our maps diversities, too.}.
In this section, we will show that a diversity can 
also be thought of as a distance on $\cP(Y)$.
This will allow us to apply the results in 
Section~\ref{sec:symmetric-functions} to 
tight-spans of diversities which we shall consider in the next section.

We first note two trivial properties of diversities; 
for a proof see \cite[Proposition~2.1]{BT10}.
\begin{lem}\label{lem:div-properties}
Let $\delta$ be a diversity on $Y$ and $A,B\subseteq Y$.
\begin{enumerate}
\item If $A\cap B\not=\emptyset$, then $\delta(A)+\delta(B)\geq \delta(A\cup B)$.
\item If $A\subseteq B$, then $\delta(A)\leq \delta(B)$.
\end{enumerate}
\end{lem}


Now, for an arbitrary symmetric map $D:\ncP(Y)\times \ncP(Y)\to \RR$, we define the following properties:
\begin{enumerate}
\item[(A1)] $D(A,B)+D(B,C)\geq D(A,C)$ for all $A,B,C\in\ncP(Y)$, \quad\text{(triangle inequality)}
\item[(A2)] $D\left(\{x\},\{x\}\right)=0$ for all $x\in Y$, and
\item[(A3)] $D(A,B)=\frac12D(A\union B,A\union B)$ for all $A,B\in\ncP(Y)$ with $A\not=B$.
\end{enumerate}

Given a map $\delta: \cP(Y) \to \RR$, let $D_\delta:\ncP(Y)\times \ncP(Y)\to \RR$ be given by
\[
D_\delta(A,B)=\begin{cases} 2\delta(A),&\text{if }A=B,\\\delta(A\cup B),&\text{else.}\end{cases}
\]
Obviously, $D_\delta$ is symmetric. Conversely, given a symmetric map $D:\ncP(Y)\times \ncP(Y)\to \RR$, we define the function $\delta(D): \cP(Y)\to \RR$ by setting $\delta(D)(A)=\frac12D(A,A)$ for $A\in\ncP(Y)$ and $\delta(D)(\emptyset)=0$.

\begin{prop}
\begin{enumerate}
\item Given an arbitrary function $\delta:\cP(Y)\to \RR$, the map $D_\delta$ satisfies (A3). Moreover, if $\delta$ is a diversity, then $D_\delta$ also satisfies (A1) and (A2).
\item If $D:\ncP(Y)\times \ncP(Y)\to \RR$ is a symmetric map fulfilling (A1)--(A3), then $\delta(D)$ is a diversity.
\item Let $\delta:\cP(Y)\to \RR$. Then $\delta(D_\delta)=\delta$.
\item Let $D:\ncP(Y)\times \ncP(Y)\to \RR$ be a symmetric map. Then $D_{\delta(D)}=D$ if and only if (A3) holds.
\end{enumerate}
\end{prop}

\begin{proof}
\begin{enumerate}
\item For all $A,B\in\ncP(Y)$ with $A\not=B$ we have $D_\delta(A,B)=\delta(A\union B)=\frac12D_\delta(A\union B,A\union B)$, that is, (A3) holds. If in addition $\delta$ is a diversity, then for all $A,B,C\in\ncP(Y)$  with $A\not=C$ we have
\[
D_\delta(A,B)+D_\delta(B,C)\geq\delta(A\cup B)+\delta(B\cup C) \overset{(D1)}{\geq} \delta(A\cup C)=D_\delta(A,C)\,.
\]
So, if $A=C$, using Lemma~\ref{lem:div-properties}, we get 
\[
D_\delta(A,B)+D_\delta(B,C)=2D_\delta(A,B)=2\delta(A\cup B) \geq 2 \delta(A)=D_\delta(A,C)\,.
\]
Hence (A1) holds. Furthermore, $D_\delta(\left(\{x\},\{x\}\right)=2\delta(\{x\})=0$ for all $x\in Y$ by (D2), so (A1) and (A2) also hold.
\item We have
\begin{align*}
\delta(D)(A\cup B)+\delta(D)(B\cup C)&=\frac12D(A\cup B,A\cup B)+\frac12D(B\cup C, B\cup C)\\
&\overset{(A3)}=D(A,B)+D(B,C)\overset{(A1)}\geq D(A,C)\\
&=\frac12D(A\cup C,A\cup C)=\delta(D)(A\cup C)\,,
\end{align*}
which is (D1). From (A2), we conclude that $\delta(D)(\{x\})=D(\{x\},\{x\})=0$ for all $x\in Y$, hence (D2) holds.
\item By definition, $\delta(D_\delta)(A)=\frac12D_\delta(A,A)=\frac22\delta(A)=\delta(A)$ for all $A\in\cP(Y)$.
\item For all $A,B\in\ncP(Y)$ with $A\not=B$ we have
\[
  D_{\delta(D)}(A,B)=\delta(D)(A\cup B)=\frac12D(A\cup B,A\cup B)\,
\]
and $D_{\delta(D)}(A,A)=2\delta(D)(A)=D(A,A)$. Hence $D_{\delta(D)}=D$ if and only if condition (A3) is satisfied.
\end{enumerate}
\end{proof}

\begin{cor}\label{cor:div-sym}
Diversities on $Y$ are in bijective correspondence with symmetric maps $\ncP(Y)\times \ncP(Y)\to \RR$ satisfying (A1)--(A3).
\end{cor}

To not only obtain a symmetric map, but even a distance on $\ncP(Y)$, one can use the process described in Section~\ref{sec:symmetric-functions} to arrive at the distance map $d_\delta:\ncP(Y)\times \ncP(Y)\to \RR$ defined by $d_\delta=((D_\delta)')_+$, or, equivalently, define
\begin{align}\label{eq:diversity-distance}
d_\delta(A,B)=\begin{cases} \max(0,\delta(A\cup B)-\left(\delta(A)+\delta(B)\right)),&\text{if }A\not=B,\\0,&\text{if }A=B\,.\end{cases}
\end{align}

\begin{lem}
The mapping from the set of diversities on $Y$ to the set of all distances on $\ncP(Y)$ defined by $\delta\mapsto d_\delta$ is injective.
\end{lem}
\begin{proof}
Let $\delta,\delta'$ be diversities on $Y$ with $d_\delta(A,B)=d_{\delta'}(A,B)$ for all $A,B\in\ncP(Y)$. We will show that $\delta(A)=\delta'(A)$ implies $\delta(A\cup\{i\})=\delta'(A\cup\{i\})$ for all $i \in Y$, which will prove the claim by induction. Since we have
\[
\max(0,\delta(A\cup\{i\})-\delta(A)-\delta(\{i\}))=d_\delta(A,\{i\})=d_{\delta'}(A,\{i\})=\max(0,\delta'(A\cup\{i\})-\delta'(A)-\delta'(\{i\}))\,,
\]
by Lemma~\ref{lem:div-properties} we also have $\delta(A\cup\{i\})-\delta(A)=\delta'(A\cup\{i\})-\delta'(A)$, which completes the proof.
\end{proof}

We conclude with one final observation which is of independent interest.

\begin{lem}
Let $Y$ be a finite set, $\delta$ a diversity on $Y$ and $d_\delta$ the associated distance on $\ncP(Y)$. Then $d_\delta(A,B)=0$ for all $A,B\in\ncP(Y)$ with $A\cap B\not=\emptyset$.
\end{lem}
\begin{proof}
By Lemma~\ref{lem:div-properties}, we have $\delta(A\cup B)-\delta(A)-\delta(B)\leq 0$ which implies $d_\delta(A,B)=0$ by the definition of $d_\delta$.
\end{proof}

\section{Tight-Spans of Diversities}\label{sec:div-ts}

In \cite{BT10}, the tight-span of a diversity is introduced, 
which is defined as follows.
Given a diversity $\delta$ on a finite set $Y$,  
the tight-span of $\delta$ is the set $\divTS \delta$ 
of all minimal elements of the polyhedron
\[
\divP \delta=\SetOf{f\in \RR^{\cP(Y)}}{f(\emptyset)=0 \text{ and } \sum_{A\in\cA} f(A)\geq \delta\left(\bigcup \cA\right) \text{ for all } \cA\subseteq\cP(Y)}\,.
\]
We shall also consider the related set 
$\divTSm{\delta}$ consisting of all minimal elements of the polyhedron
\[
\divPm{\delta} =\SetOf{f\in \RR^{\cP(Y)}}{f(\emptyset)=0 \text{ and } f(A)+f(B)\geq \delta(A\cup B) \text{ for all } A,B\in\cP(Y)}\,
\]
which was mentioned in a preliminary version of 
\cite{BT10}. Note that Bryant and Tupper 
showed that the set $\divTS \delta$ is an injective 
hull in an appropriately defined category. Although 
a similar result has not be shown to hold for
$\divTSm{\delta}$, this set will also be useful in 
our investigations of diversities.

We first prove that $\divTSm{\delta}$ 
arises from the point configuration $\cA(\ncP(Y))$, 
the configuration of integer points 
in a simplex with edge-length 2, 
and that $\divTS\delta$ arises from $\cC(\ncP(Y))$, 
the configuration of vertices of a cube.
More specifically, let  $w^\delta: \cC(\ncP(Y))\to \RR$ be 
given by $w^\delta(\sum_{i\in \cA} e_i)=\delta(\cup\cA)$ 
for all $\cA\subseteq\ncP(Y)$. Then:

\begin{prop}\label{prop:divts-pc}
Let $\delta: \cP(Y)\to \RR$. Then we have:
\begin{enumerate}
\item $\divPm{\delta} =\{0\}\times\envelope{w^{D_\delta}}{\cA(\ncP(Y)}$, and
$\divTSm{\delta} =\{0\}\times\tightspan{w^{D_\delta}}{\cA(\ncP(Y))}$.
\item $\divP{\delta} =\{0\}\times\envelope{w^\delta}{\cC(\ncP(Y)}$, and
$\divTS{\delta} =\{0\}\times\tightspan{w^\delta}{\cC(\ncP(Y))}$.
\end{enumerate}
\end{prop}

\begin{proof}
\begin{enumerate}
\item We have 
\begin{align*}
\divPm{\delta}&=\SetOf{f\in \RR^{\cP(Y)}}{f(\emptyset)=0 \text{ and } f(A)+f(B)\geq \delta(A\cup B) \text{ for all } A,B\in\cP(Y)}\\
&=\left\{f\in \RR^{\cP(Y)}\,\right|\,f(\emptyset)=0, f(A)+f(\emptyset)=f(A)\geq \delta(A) \text{ for all } A\in\ncP(Y) \\
&\qquad \text{ and } f(A)+f(B)\geq \delta(A\cup B) \text{ for all } A,B\in\ncP(Y)\Big\}\\
&=\{0\}\times\left\{f\in \RR^{\ncP(Y)}\,\right|\,f(A)+f(A)\geq 2\delta(A) \text{ for all } A\in\ncP(Y), f(A)+f(A)\geq \delta(A\cup A)\\
 &\qquad \text{ for all } A\in\ncP(Y), \text{ and } f(A)+f(B)\geq \delta(A\cup B) \text{ for all } A,B\in\ncP\text{ with }A\not=B\Big\}\\
&=\{0\}\times \envelope{w^{D_\delta}}{\cA(\ncP(Y)}\,.
\end{align*}
The statement now follows from Proposition~\ref{prop:ts-dissimilarity}.
\item 
We have 
\begin{align*}
\divP \delta&=\SetOf{f\in \RR^{\cP(Y)}}{f(\emptyset)=0 \text{ and } \sum_{A\in\cA} f(A)\geq \delta\left(\bigcup \cA\right) \text{ for all } \cA\subseteq\cP(Y)}\\
&=\{0\}\times\SetOf{f\in \RR^{\ncP(Y)}}{\scp{\sum_{A\in\cA}e_A}{f}\geq \delta\left(\bigcup \cA\right) \text{ for all } \cA\subseteq\cP(Y)}\\
&=\{0\}\times\SetOf{f\in \RR^{\ncP(Y)}}{\scp{a}{f}\geq w^\delta(a) \text{ for all } a\in\cC(\ncP(Y))}\\
&=\{0\}\times \envelope{w^\delta}{\cC(\ncP(Y))}\,.
\end{align*}
The statement now follows from Lemma~\ref{lem:minimal} since $\cC(\ncP(Y))$ is positive and $\envelope{w^\delta}{\cC(\ncP(Y)}$ is bounded from below by the inequalities $e_A\geq \delta(A)=0$ for all $A\in\ncP(Y)$.
\end{enumerate}
\end{proof}

Note that, by considering the distance map $d_\delta$ defined in \eqref{eq:diversity-distance}, we get
\[
\divTSm{\delta}=\{0\}\times T_{d_\delta}+\left(\delta(A)\right)_{A\in\cP(X)}\,,
\]
which, in particular, implies the following:
\begin{cor}\label{cor:ts-diversity-distance}
$\divTSm{\delta}$ corresponding to the diversity $\delta$ is a tree if and only if the tight-span $T_{d_\delta}$ of the distance $d_\delta$ is a tree. 
\end{cor}

A special kind of diversity arising from a phylogenetic tree is given as follows:
Let $T$ be a weighted tree with leaf set $Y$. The \emph{phylogenetic diversity} $\delta_T$ associated to $T$ is defined by mapping any subset $A\subseteq Y$ to the length $\delta_T(A)$ of the smallest subtree of $T$ connecting taxa in $A$. These are precisely the diversities whose tight-spans are trees:

\begin{thm}\label{thm:phylogenetic-tree}
Let $\delta$ be diversity on $Y$. Then the following are equivalent:
\begin{enumerate}
\item \label{thm:phylogenetic-tree:diversity}$\delta$ is a phylogenetic diversity.
\item \label{thm:phylogenetic-tree:divTS}The tight-span $\divTS{\delta}$ is a tree.
\item \label{thm:phylogenetic-tree:divTSm}The tight-span $\divTSm{\delta}$ is a tree.
\end{enumerate}
\end{thm}

That \eqref{thm:phylogenetic-tree:diversity} implies \eqref{thm:phylogenetic-tree:divTS} follows from \cite[Theorem~5.8]{BT10}, since the tight-span of a phylogenetic diversity is isomorphic to the tight-span of the metric space associated to the same tree. To show that \eqref{thm:phylogenetic-tree:divTS} is equivalent \eqref{thm:phylogenetic-tree:divTSm}, we will prove that $\divTS{\delta}$ and $\divTSm{\delta}$ are equal for a much larger class of diversities, the so-called \emph{split-system diversities} (Theorem~\ref{thm:tight-span-equal} below). The proof that \eqref{thm:phylogenetic-tree:divTSm} implies \eqref{thm:phylogenetic-tree:diversity} will then take up the remainder of this section.

Let $S$ be a split of $Y$.  We define the \emph{split diversity} $\delta_S:\cP(Y)\to \RR$ of $S$ as
\[
\delta_S(A)=
\begin{cases}
 1,&\text{if } S\text{ splits }A,\\
0,&\text{else\,.}
\end{cases}
\]

Given a set $\cS$ of splits of $Y$ and a function $\alpha:\cS\to\RR_{>0}$ assigning weights to the splits in $\cS$, the \emph{split system diversity} $\delta_{(\cS,\alpha)}$ of $(\cS,\alpha)$ is defined as
\[
\delta_{(\cS,\alpha)}(A)=\sum_{S\in\cS}\alpha(S)\delta_S(A)=\sum_{S\in\cS, S \text{ splits }A}\alpha(S)\,.
\]
A phylogenetic diversity is a special case of a split system diversity where the set $\cS$ is compatible. We now show that in case $\delta$ is a split diversity, the tight-spans $\divTSm{\delta}$ and $\divTS{\delta}$ are equal:

\begin{thm}\label{thm:tight-span-equal}
Let $\cS$ be a split system, $\alpha:\cS\to \RR_{>0}$ and $\delta_{(\cS,\alpha)}$ the associated split system diversity. Then
\[
\divP {\delta_{(\cS,\alpha)}}=\divPm {\delta_{(\cS,\alpha)}}\text{ and }\divTS {\delta_{(\cS,\alpha)}}=\divTSm {\delta_{(\cS,\alpha)}}\,.
\]
\end{thm}
\begin{proof}
We will show that $\divP {\delta_{(\cS,\alpha)}}=\divPm {\delta_{(\cS,\alpha)}}$ which obviously implies $\divTS {\delta_{(\cS,\alpha)}}=\divTSm {\delta_{(\cS,\alpha)}}$. First note that, by definition, for any diversity $\delta$, one has $\divP {\delta}\subseteq\divPm {\delta}$ and furthermore, for all $f\in\divPm {\delta}$ and $A\in\cP(Y)$, one has $f(A)+f(\emptyset)=f(A)\geq\delta(A)$.
It now suffices to show that for any $\cA\subseteq \cP(Y)$ with $\card \cA\geq 2$ and $f\in\divPm{\delta_{(\cS,\alpha)}}$ the system of inequalities
\begin{align}\label{eq:div-2term}
f(A)+f(B)\geq \delta_{(\cS,\alpha)}(A\cup B)\quad\text{for all distinct } A,B\in\cA\,,
\end{align}
implies the inequality
\begin{align}\label{eq:div-allterm}
 \sum_{A\in\cA} f(A)\geq \delta_{(\cS,\alpha)}\left(\bigcup \cA\right)\,.
\end{align}
Summing up all the Inequalities~\eqref{eq:div-2term} we get
\begin{align*}
\sum_{A,B\in\cA,A\not=B}\left(f(A)+f(B)\right)&\geq\sum_{A,B\in\cA,A\not=B}\delta_{(\cS,\alpha)}(A\cup B)\quad\iff\\
\left(\card \cA-1\right)\sum_{A\in\cA}f(A)&\geq\sum_{A,B\in\cA,A\not=B}\sum_{S\in\cS, S \text{ splits }A\cup B}\alpha(S)=\sum_{S\in\cS}\alpha(S)\SP(S,\cA,2)\,,
\end{align*}
where $\SP(S,\cA,2)$ denotes the number of unordered pairs of distinct $A,B\in\cA$ such that $S$ splits $A\cup B$. We now show that $\SP(S,\cA,2)\geq \card\cA-1$ for all $S\in \cS$ that split $\cup \cA$. Dividing the above inequality by $\card\cA-1$ then gives Inequality~\eqref{eq:div-allterm} as desired.

Let $S\in\cS$ be a split that splits $\cup \cA$. First suppose that there exists some $A\in\cA$ that is split by $S$. Then $\SP(S,\cA,2)\geq \card\cA-1$ since obviously for all $B\in\cA$ distinct from $A$ the set $A\cup B$ is also split by $S$. So we can assume that, for $S=(C,D)$ and all $A\in\cA$, we have either $A\subseteq C$ or $A\subseteq D$. Let $s$ be the number of $A\in\cA$ with $A\subseteq C$. Since $S$ splits $\cup \cA$, both $s$ and $\card\cA-s$ have to be at least $1$, so we get $\SP(S,\cA,2)=s(\card\cA-s)\geq \cA-1$, which finishes the proof.
\end{proof}

In preparation to complete the proof of Theorem~\ref{thm:phylogenetic-tree}, we first examine the distance $d_\delta$ in case $\delta$ is a split diversity. So, let $S=\{C,D\}$ be a split of $Y$. For each $A,B\subseteq Y$ we have
\[
\delta_S(A\cup B)-\delta_S(A)-\delta_S(B)=
\begin{cases}
 1,&\text{if } C\subseteq A\text{ and } D\subseteq B\text{ or } C\subseteq B\text{ and } D\subseteq A\\
-1,&\text{if } A\cap C\not=\emptyset, A\cap D\not=\emptyset, B\cap C\not=\emptyset,\text{ and } B\cap D\not=\emptyset,\\
0,&\text{else\,,}
\end{cases}
\]
so that
\[
d_{\delta_S}(A,B)=\begin{cases} 1,&\text{if } C\subseteq A\text{ and } D\subseteq B\text{ or } C\subseteq B\text{ and } D\subseteq A,\\
0,&\text{else.}\end{cases}
\]
Hence $d_{\delta_S}=d_P$ (as defined in Equation~\eqref{eq:partial-distance}), where $P$ is the partial split $\{\ncP(A),\ncP(B)\}$ of $\ncP(Y)$. By Proposition~\ref{prop:splits-a}, the corresponding weight function defines a split of $\cA(\ncP(Y))$. So as to show that a diversity whose tight-span is a tree comes from a split system, we consider these steps in reverse order. The following lemmas will be the key to our proof.

\begin{lem}\label{lem:div2}
Let $Y$ be a finite set, $\{\cA,\cB\}$, $\{\cC,\cD\}$ two compatible partial splits of $
\cP(Y)$ and $A\in\cA$, $A'\in\cC$, $B\in\cB\cap\cD$. Then either $A'\in\cA$ or $A\in\cC$.
\end{lem}
\begin{proof}
Since $\{\cA,\cB\}$, $\{\cC,\cD\}$ are compatible and $B\in\cB\cap\cD$, by definition, we must have $\cA\subseteq \cC$ or $\cC\subseteq \cA$, which implies the claim.
\end{proof}

\begin{lem}\label{lem:div-tree-props}
Let $\cP$ be a compatible set of partial splits of $\ncP(Y)$, $\{\cA,\cB\}\in\cP$, $\alpha:\cP\to\RR_{> 0}$ and $\delta$ be a diversity on $Y$ such that $d_\delta=d_{(\cP,\alpha)}$. Then we have:
\begin{enumerate}
\item \label{lem:div-tree-props-basic}If $A\in \cA$, $B\in\cB$, $A'\subseteq A$, and $d_\delta(A',B)\geq d_\delta(A,B)$, then $A'\in \cA$.
\item \label{lem:div-tree-props:singleton}If $A\in\cA$, then $\{i\}\in\cA$ for all $i\in A$.
\item \label{lem:div-tree-props:0}If $A,A'\in\cA$, then $d_\delta(A,A')=0$.
\item \label{lem:div-tree-props:add}If $i\not\in A\in\cP(Y)$ with $\{i\},A\in\cA$, then $\{i\}\cup A\in\cA$.
\end{enumerate}
\end{lem}
\begin{proof}
\begin{enumerate}
\item For $A,B\subseteq Y$, set $\cP(A,B)=\smallSetOf{\{\cA,\cB\}\in\cP}{A\in\cA,B\in\cB}$. Then we have
\begin{align*}
d_\delta(A',B)&=d_{(\cP,\alpha)}(A',B)=\sum_{P\in \cP(A',B)} \alpha(P)\\
&\geq\sum_{P\in\cP(A,B)}=d_\delta(A,B)\,.
\end{align*}
Then either $\cP(A,B)=\cP(A',B)$, which trivially implies the claim, or there exists some partial split $\{\cC,\cD\}\in \cP(A',B)\setminus\cP(A,B)$.  Lemma~\ref{lem:div2} now gives us $A'\in\cA$ (since $A\in C$ would imply $\{\cC,\cD\}\in \cP(A,B)$, as desired).
\item Let $B\in\cB$. By (D1), we get
\[
\delta(A\cup\{i\})+\delta(\{i\}\cup B)\geq \delta(A\cup B)\,.
\]
This implies that
\[
\delta(\{i\}\cup B)-\delta(B)\geq \delta(A\cup B)-\delta(A)-\delta(B)\,.
\]
By definition of $d_\delta$ (and (D2)), this is equivalent to $d_\delta(\{i\},B)\geq d_\delta(A,B)$. The claim now follows from Part~\eqref{lem:div-tree-props-basic}.

\item If $d_\delta(A,A')\not=0$ this would imply that there exists a  partial split $\{\cC,\cD\}\in\cP$ with $A\in\cC$ and $A'\in\cD$ (or vice versa). However, this split could not be compatible with $\{\cA,\cB\}$.
\item By Part~\eqref{lem:div-tree-props:0}, we have $d_\delta(\{i\},A)=0$ which implies $\delta(A\cup \{i\})=\delta(A)$ by the definition of $d_\delta$. Let $B\in\cB$. Since $d_\delta(A,B)>0$ by assumption, we get
\begin{align*}
d_\delta(A,B)&=\delta(A\cup B)-\delta(A)-\delta(B)\\
&=\delta(A\cup B)-\delta(A\cup\{i\})-\delta(B)\\
&\leq\delta(A\cup \{i\}\cup B)-\delta(A\cup\{i\})-\delta(B)\\
&=d_\delta(A\cup\{i\},B)\,.
\end{align*}
The claim now follows from Part~\eqref{lem:div-tree-props-basic}.
\end{enumerate}
\end{proof}

\begin{cor}\label{cor:div-tree-props}
Let $\cP$ be a compatible set of partial splits of $\ncP(Y)$, $\alpha:\cP\to\RR_{\geq 0}$ and $\delta$ a diversity on $Y$ such that $d_\delta=d_{(\cP,\alpha)}$. Then for all $P=\{\cA,\cB\}\in\cP$ there exists a partial split $p(P)=\{A,B\}$ of $Y$ such that $\cA=\ncP(A)$ and $\cB=\ncP(B)$. Furthermore the set $\smallSetOf{p(P)}{P\in\cP}$ of partial splits of $Y$ is compatible.
\end{cor}
\begin{proof}
The existence of $p(P)$ follows by iteratively applying Lemma \ref{lem:div-tree-props} \eqref{lem:div-tree-props:singleton} and  \eqref{lem:div-tree-props:add}. The compatibility follows from the compatibility of $\cP$ by Proposition~\ref{prop:partial-splits-comp}.
\end{proof}

We can now finish the proof of the main theorem in this section.

\begin{proof}[Proof of Theorem~\ref{thm:phylogenetic-tree}]
We only have to show that \eqref{thm:phylogenetic-tree:divTSm} implies \eqref{thm:phylogenetic-tree:diversity}. So, let $\delta$ be a diversity on $Y$ such that the tight-span $\divTSm{\delta}$ is a tree. Corollary~\ref{cor:ts-diversity-distance} now implies that $T_{d_\delta}$ is a tree and Theorem~\ref{thm:distance-tree} gives us a compatible set $\cP$ of partial splits of $\ncP(Y)$ and a function $\alpha:\cP\to \RR_{>0}$ such that $d_\delta=d_{(\cP,\alpha)}$. By Corollary \ref{cor:div-tree-props}, there exists a compatible set $\cS=\smallSetOf{p(P)}{P\in\cP}$ of partial splits of $Y$ such that $\cA=\ncP(A)$ and $\cB=\ncP(B)$. It remains to show that all $S\in\cS$ are splits.

Suppose one of these partial splits, say $\{A,B\}$, is not a split, and let $i\in A$, $j\in B$ and $k\in Y\setminus (A\cup B)$. By (D1), we get
\[
\delta(\{i,k\})+\delta(\{k,j\})\geq\delta(\{i,j\})\,,
\]
which is equivalent to
\begin{align}\label{eq:div-proof}
d_\delta(\{i\},\{k\})+d_\delta(\{k\},\{j\})\geq d_\delta(\{i\},\{j\})
\end{align}
by the definition of $d_\delta$ and (D2). Now any partial split $\{\ncP(C),\ncP(D)\}$ that separates $k$ from either $i$ or $j$ must also separate $i$ and $j$ since it is compatible to $\{\ncP(A),\ncP(B)\}$. So each split making a contribution to the left of Equation \eqref{eq:div-proof} makes the same contribution to the right of the equation and $\{\ncP(A),\ncP(B)\}$ only contributes to the right, a contradiction.

So $d_\delta=d_{\delta(\cS,\alpha)}$ which implies $\delta=\delta_{(\cS,\alpha)}$. Hence $\delta$ is a phylogenetic diversity.
\end{proof}


\section{Discussion}\label{sec:discussion}

\subsection{$k$"=Dissimilarity Maps}

We have seen how to define the tight-span of various
maps that generalise metrics. Another kind of map that
we could consider taking the tight-span of
is a \emph{$k$"=dissimilarity map} on a set $X$,
that is, a function $D:\binom Xk \to \RR$. In this case,
to obtain a tight-span
one could take the set of vertices of the hypersimplex $\Hypersimplex k X\subseteq \RR^X$
as the corresponding point configuration, that is, the set of
all functions $\sum_{x\in A} e_x$ for all $A\in \binom Xk$. More specifically, motivated by
Proposition \ref{prop:ts-dissimilarity} for
the case $k=2$, given a $k$"=dissimilarity map $D$ on $X$, define the
function $w^D: \Hypersimplex k X\subseteq \RR^X\to \RR$
that sends $\sum_{x\in A} e_x$ to $-D(A)$,
set $P_D=\envelope{w^D}{\Hypersimplex kX}$,
and $T_D=\tightspan{w^D}{\Hypersimplex kX}$.
It follows from Lemma~\ref{lem:minimal} that $T_D$
is the set of minimal elements of $P_D$.

Even though one might expect that $T_D$ has 
similar properties to the tight-spans we have
so far considered, this is not the case.
Indeed, given a weighted tree $T$ with leaf set $X$,
one can define a $k$"=dissimilarity map $D^k_T$ by assigning
to each $k$"=subset $A\subseteq X$ the total length of the induced subtree.
However, the tight-span $T_{D^k_T}$
does not in general also have to be a tree, and
so there is no obvious generalisation of the Tree Metric Theorem.
Even so, the tree $T$ can be reconstructed from $T_{D^k_T}$
\cite[Section~8.1]{herr-moul-10}, and so it could still
be of interest to further study these tight-spans.

\subsection{Coherent decompositions}

Coherent decompositions of metrics were introduced by
Bandelt and Dress \cite{MR1153934} and are intimately
related to tight-spans. Thus the question arises whether
a similar decomposition theory could be
developed for the different generalisations of metrics that
we have considered.

In \cite{MR2502496}, the concept of coherent decompositions
of metrics was generalised to weight functions of
polytopes (as discussed in Section~\ref{sec:point-configurations}).
For directed distances, and also symmetric and non-symmetric
functions, this directly leads to a theory of
coherent decompositions. Moreover, a decomposition
theorem for $k$"=dissimilarities in terms of ``split
$k$"=dissimilarities''  was recently derived
in~\cite{herr-moul-10}, which might be extended if
an appropriate theory was worked out for
tight-spans of  $k$"=dissimilarities, as suggested above.
For diversities, as we have seen, a diversity on a set $Y$ can be
considered as a distance on $\ncP(Y)$,
but such diversities form only a subset of all such distances. So to
develop a theory of coherent decomposition for diversities,
one could maybe try to first answer the following question:
\begin{qst}
Given a diversity $\delta$ how can one compute coherent
decompositions of the distance $d_\delta$? Moreover, which
coherent components  of such a decomposition are of
the form $d_{\delta^\star}$ for some diversity $\delta^\star$?
\end{qst}

\subsection{Infinite Sets and Injective Hulls}

In this paper, we have only considered tight-spans arising from finite sets.
However, many of the results
concerning tight-spans (and not point configurations) can be
translated to infinite sets. For example,
much of the theory for the tight-span of a metric
space was originally developed for arbitrary metric spaces \cite{MR753872,MR82949}, which is
important since the tight-span of a metric (diversity) ---
which is of course an infinite set --- comes
equipped with a canonical metric (diversity), such that the
tight-span of this metric is nothing other
than itself (see \cite{MR82949,BT10}, respectively).
Stated differently, this means that 
tight-spans are injective objects in the appropriate category
\cite{MR82949}, a property that would
be interesting to understand in the setting of
point configurations. However, if the theory
for point configurations is to be extended to
infinite sets, a first crucial step would
be to understand how to generalise splits of polytopes,
which appears to have no obvious generalisation in the infinite setting.\\[5pt]
\noindent{\bf Acknowledgement:} The authors thank the anonymous referees
for their helpful comments. 

\bibliographystyle{amsplain}
\bibliography{tight_spans}

\end{document}